\documentclass[11pt]{article}
\usepackage[obeyspaces,hyphens,spaces]{url}
\usepackage{framed} 

\usepackage{amsthm,amsmath,amsfonts,mathtools,wasysym,empheq} 
\usepackage{amssymb} 

\usepackage{mathpazo}

\usepackage[mathscr]{eucal} 

\usepackage[top= 2 cm, bottom = 2 cm, left = 2 cm, right= 2 cm]{geometry}
\usepackage[title]{appendix}

\usepackage[table, dvipsnames]{xcolor} %

\usepackage[labelsep=space]{caption}  

\usepackage{graphicx, soul}
\usepackage{float,booktabs,multirow}
\restylefloat{table*}

\usepackage{todonotes}

\definecolor{labelkey}{rgb}{.1,.1,.8}
\definecolor{refkey}{rgb}{0,0.6,0.0}

\usepackage[shortlabels]{enumitem}
\setlist[enumerate]{itemsep = 0.5pt,topsep=2pt}

\definecolor{dgreen}{rgb}{0.00,0.49,0.00}
\definecolor{dblue}{rgb}{0,0.08,0.75}
\colorlet{myblue}{dblue}
\colorlet{mygreen}{dgreen}
 
\definecolor{myfirstblue}{rgb}{.8, .8, 1}

\newcommand*\mybluebox[1]{%
    \colorbox{RoyalBlue!20}{\hspace{1em}#1\hspace{1em}}}

\usepackage{hyperref}
\hypersetup{ 
	linktocpage=true,
	colorlinks= true, 
	urlcolor = magenta, 
	citecolor = mygreen,
	linkcolor = myblue
}

\allowdisplaybreaks

\usepackage[capitalize, nameinlink]{cleveref} 

\crefdefaultlabelformat{#2\textup{#1}#3}
\crefrangeformat{enumi}{#3\textup{#1}#4\textup{\textendash}#5\textup{#2}#6}
\crefname{equation}{}{}

\crefname{chapter}{Appendix}{chapters}
\crefname{item}{}{items}
\crefname{figure}{Figure}{Figures}
\crefname{theorem}{\protect\theoremname}{Theorems}
\crefname{lemma}{\protect\lemmaname}{\protect\lemmaname}
\crefname{proposition}{\protect\propositionname}{\protect\propositionname}
\crefname{corollary}{\protect\corollaryname}{\protect\corollaryname}
\crefname{definition}{\protect\definitionname}{\protect\definitionname}
\crefname{fact}{\protect\factname}{\protect\factname}
\crefname{example}{\protect\examplename}{\protect\examplename}
\crefname{algorithm}{Algorithm}{Algorithms}
\crefname{remark}{\protect\remarkname}{\protect\remarkname}
\crefname{case}{\protect\casename}{\protect\casename}
\crefname{question}{\protect\questionname}{\protect\questionname}
\crefname{claim}{\protect\claimname}{\protect\claimname}
\crefname{enumi}{}{}
\crefname{appsec}{Appendix}{Appendices}

\makeatletter

\g@addto@macro\normalsize{%
  \setlength\abovedisplayskip{6pt}
  \setlength\belowdisplayskip{6pt}
  \setlength\abovedisplayshortskip{6pt}
  \setlength\belowdisplayshortskip{6pt}
}

\let\orgdescriptionlabel\descriptionlabel
\renewcommand*{\descriptionlabel}[1]{%
	\let\orglabel\label
	\let\label\@gobble
	\phantomsection
	\edef\@currentlabel{#1}%
	\let\label\orglabel
	\orgdescriptionlabel{#1}%
}

\let\geq\geqslant

\def\th@plain{%
	\thm@notefont{} 
	\itshape 
}
\def\th@definition{%
	\thm@notefont{}
	\normalfont 
}

\g@addto@macro\th@remark{\thm@headpunct{}}
\g@addto@macro\th@definition{\thm@headpunct{}}
\g@addto@macro\th@plain{\thm@headpunct{}}

\makeatother

\theoremstyle{plain}

\newtheorem{theorem}{\protect\theoremname}[section]

\newtheorem{proposition}[theorem]{\protect\propositionname}

\theoremstyle{definition}

\newtheorem{remark}[theorem]{\protect\remarkname}

\newtheorem{example}[theorem]{\protect\examplename}

\newtheorem{fact}[theorem]{\protect\factname}

\theoremstyle{remark}

\providecommand{\theoremname}{Theorem}
\providecommand{\propositionname}{Proposition}
\providecommand{\corollaryname}{Corollary}
\providecommand{\factname}{Fact}
\providecommand{\lemmaname}{Lemma}
\providecommand{\assumptionname}{Assumption}
\providecommand{\algorithmname}{Algorithm}

\providecommand{\definitionname}{Definition}
\providecommand{\notationname}{Notation}
\providecommand{\remarkname}{Remark}

\providecommand{\examplename}{Example}
\providecommand{\claimname}{Claim}

\providecommand{\algorithmname}{Algorithm}
\providecommand{\openprobname}{Open Problem}

\let\originalleft\left
\let\originalright\right
\renewcommand{\left}{\mathopen{}\mathclose\bgroup\originalleft}
\renewcommand{\right}{\aftergroup\egroup\originalright}

\DeclarePairedDelimiter{\norm}{\lVert}{\rVert}
\DeclarePairedDelimiterX{\scal}[2]{\langle}{\rangle}{  #1 \, \delimsize \vert \, \mathopen{}  #2  }

\DeclarePairedDelimiterX{\fnorm}[1]{\lVert}{\rVert_\ensuremath{\mathsf{F}}}{#1}

\DeclarePairedDelimiterX\menge[2]{ \{ }{ \} }{ {#1} ~ \delimsize \vert ~ \mathopen{}  {#2} }  
\DeclarePairedDelimiterX\fa[2]{ ( }{ )_{#2} }{#1}  
\DeclarePairedDelimiterX\set[2]{ \{ }{ \}_{#2} }{#1}  

\newcommand{\mmenge}[2]{\bigg\{{#1}~\bigg |~{#2}\bigg\}}

\newcommand{\NN}{\ensuremath{\mathbb N}}

\newcommand{\minf}{\ensuremath{ {-}\infty}}
\newcommand{\pinf}{\ensuremath{ {+}\infty}}
\newcommand{\RX}{\ensuremath{ \left] \minf, \pinf \right] }}

\newcommand{\zeroun}{\ensuremath{\left]0,1\right[}}

\newcommand{\Fix}{\ensuremath{\operatorname{Fix}}}
\newcommand{\Id}{\ensuremath{\operatorname{Id}}}
\newcommand{\prox}[1]{\ensuremath{\mathop{\operatorname{Prox}_{#1}}}}

\newcommand{\zer}{{\ensuremath{\operatorname{zer\,}}}} 
\newcommand{\weakly}{{\ensuremath{\,\rightharpoonup\,}}} 
\newcommand{\To}{{\ensuremath{\,\rightrightarrows\,}}} 

\newcommand{\email}[1]{\href{mailto:#1}{\nolinkurl{#1}}} 

\DeclareSymbolFont{italics}{\encodingdefault}{\rmdefault}{m}{it}
\DeclareMathSymbol{f}{\mathalpha}{italics}{`f}

\let\oldFootnote\footnote
\newcommand\nextToken\relax

\renewcommand\footnote[1]{%
    \oldFootnote{#1}\futurelet\nextToken\isFootnote}

\newcommand\isFootnote{%
    \ifx\footnote\nextToken\textsuperscript{,}\fi}

\begin{document}

\title{ \sffamily  
  On the asymptotic behaviour of the
  Arag\'on Artacho--Campoy algorithm
 }

\author{
  	Salihah Alwadani\thanks{
  		Mathematics, University of British Columbia, Kelowna, B.C.\ V1V~1V7, Canada. 
  		Email: \href{mailto:
		salihahalwadani@hotmail.com}{\texttt{salihahalwadani@hotmail.com}}.},~
  	Heinz H.\ Bauschke\thanks{
  		Mathematics, University of British Columbia, Kelowna, B.C.\ V1V~1V7, Canada. 
  		Email: \href{mailto: heinz.bauschke@ubc.ca}{\texttt{heinz.bauschke@ubc.ca}}.},~
  	Walaa M. Moursi\thanks{Electrical Engineering, 
    Stanford University, Stanford, CA 94305, USA and
    Mansoura University, Faculty of Science, Mathematics
    Department, Mansoura 35516, Egypt. 
  Email: \href{mailto:
  wmoursi@stanford.edu}{\texttt{wmoursi@stanford.edu}}.},~
  	and Xianfu Wang\thanks{
  	Mathematics, University of British Columbia, Kelowna, B.C.\ V1V~1V7, Canada. 
  	Email: \href{mailto: shawn.wang@ubc.ca}{\texttt{shawn.wang@ubc.ca}}.}
}

\date{May~28, 2018}


\maketitle

\begin{abstract}
\noindent
Arag\'on Artacho and Campoy recently proposed a new method
for computing the projection onto the intersection of two closed
convex sets in Hilbert space; moreover, they proposed in 2018 
a generalization from normal cone operators to maximally monotone
operators. 
In this paper, we complete this analysis by demonstrating that 
the underlying curve converges to the nearest zero of the
sum of the two operators. 
We also provide a new interpretation of the underlying operators
in terms of the resolvent and the proximal average. 
\end{abstract}
{\small

}


\section{Introduction}
Throughout this note, 
\begin{empheq}[box=\mybluebox]{equation}
    \label{eq:Hilbert.space}
    \text{$X$ is a real Hilbert space}
\end{empheq}
with inner product $\scal{\, \cdot}{\cdot\,} $
and associated norm $\norm{\, \cdot \, } $.
The notation of our paper 
is standard and mainly follows \cite{BC} to which we also
refer to basic results on convex analysis and monotone operator
theory. 
A central problem is to find a zero (critical point) of the sum
of two maximally monotone operators. 
The Douglas--Rachford and Peaceman--Rachford algorithms (see
Fact~\ref{f:DRPR} below) are classical approaches to solve this
problem. If the monotone operators are normal cone operators of
closed convex nonempty subsets of $X$, then one obtains a
feasibility problem. Suppose, however, that we are interested in
finding the nearest point in the intersection. One may then apply
several classical best approximation 
algorithms (see, e.g., \cite[Chapter~30]{BC}).
In the recently published paper \cite{Fran1}, Arag\'on
Artacho and Campoy presented a novel algorithm, which we term the
\emph{Aarag\'on Artacho--Campoy Algorithm (AACA)} to solve this best
approximation problem. Even more recently, they extended this
algorithm in \cite{Fran2} 
to deal with general maximally monotone operators. 

{\em The aim of this paper is to re-derive the AACA from the view
point of the proximal and resolvent average. We also complete
their analysis by describing the asymptotic behaviour of the
underlying curve.}

This note is organized as follows. 
In Section~\ref{sec:aux}, we collect a few facts and results that
will make the subsequent analysis more clear. 
Section~\ref{sec:main} contains a new variant
of a convergence result for AACA (Theorem~\ref{t:AACA}) as well
as the announced asymptotic behaviour of the curve
(Theorem~\ref{t:dichotomy}).

\section{Auxiliary results}

\label{sec:aux}

\begin{fact}[Douglas--Rachford and Peaceman--Rachford]
\label{f:DRPR}
Let $A$ and $B$ be maximally monotone on $X$. 
Suppose that 
$\zer(A+B) = (A+B)^{-1}(0)\neq\varnothing$, let 
$\lambda\in\left]0,1\right]$, and set 
\begin{equation}
T = (1-\lambda)\Id + \lambda R_BR_A, 
\end{equation}
where $J_A = (\Id+A)^{-1}$ and $R_A = 2J_A-\Id$. 
Let $x_0\in X$ and define
\begin{equation}
(\forall n\in\NN)\quad x_{n+1} = Tx_n.
\end{equation}
Then 
there exists $\bar{x}\in\Fix T$ such that
$\bar{z} = J_A\bar{x}\in \zer(A+B)$ and the following hold:
\begin{enumerate}
\item 
If $A$ or $B$ is strongly monotone, then 
$\zer(A+B) = \{\bar{z}\}$. 
\item
If $\lambda <1$, then 
$x_n\weakly \bar{x}$ and $J_Ax_n\weakly \bar{z}$. 
\item 
If $\lambda<1$ and 
$A$ or $B$ is strongly monotone, then 
$J_Ax_n\to \bar{z}$. 
\item
If $\lambda=1$ and $A$ is strongly monotone,
then $J_Ax_n\to\bar{z}$. 
\end{enumerate}
\end{fact}
\begin{proof}
This follows from \cite[Theorem~26.11 and Proposition~26.13]{BC}.
See also \cite{PLC09}. 
\end{proof}

The proof of the following result, which is a slight
generalization of \cite[Proposition~3.1]{Fran2}, 
is straight-forward and hence omitted. 
\begin{proposition}
\label{p:resav}
Let $C$ be maximally monotone on $X$, let $w\in X$, 
let 
$\gamma \in\left]0,1\right]$, and 
set
\begin{equation}
C_\gamma \colon X\To X\colon x\mapsto 
C(\gamma^{-1}(x-(1-\gamma)w))+(1-\gamma)\gamma^{-1}(x-w).
\end{equation}
Then  $C_\gamma$ is maximally monotone and its resolvent is given
by 
\begin{equation}
J_{C_\gamma}\colon X\to X\colon x\mapsto  \gamma J_{C}x+(1-\gamma)w.
\end{equation}
\end{proposition}

\begin{remark}[resolvent and proximal average]
\label{r:reproav}
Consider the setting of Proposition~\ref{p:resav}. 
Because $J_{C_\gamma}$ is a convex combination of 
the resolvents $J_C$ and $P_{\{w\}}$, we see that
$C_\gamma$ is nothing but a \emph{resolvent average}
of $C$ and $N_{\{w\}}$.
See \cite{BBMW} for a detailed study of resolvent averages.
We note that if $C$ is $\sigma_C$-monotone, i.e., $C-\sigma_C\Id$
is monotone, then 
\begin{equation}
\label{e:0525a}
\text{$C_\gamma$ is $\gamma^{-1}(\sigma_C+1-\gamma)$-monotone.}
\end{equation}
This can be verified directly (as in \cite[Proposition~3.1]{Fran2}) 
or it also follows from \cite[Theorem~3.20]{BBMW}. 

Now suppose that additionally $C=\partial h$ for some proper lower
semicontinuous convex function $h$ on $X$.
Then $C_\gamma = \partial h_\gamma$ and $J_{C_\gamma}= \prox{h_\gamma}$, 
where 
\begin{subequations}
\begin{align}
h_\gamma
\colon X&\to\RX\\
x&\mapsto 
\inf\mmenge{\gamma
h(y_1)+(1-\gamma)\iota_{\{w\}}(y_2)+\frac{\gamma(1-\gamma)}{2}\|y_1-y_2\|^2}{\gamma
y_1+(1-\gamma)y_2=x}\\
&= \gamma
h\big(\gamma^{-1}(x-(1-\gamma)w)\big)+
\frac{\gamma(1-\gamma)}{2}\|\gamma^{-1}(x-(1-\gamma)w)-w\|^2\\
&= \gamma
h\big(\gamma^{-1}(x-(1-\gamma)w)\big)+
\frac{1-\gamma}{2\gamma}\|x-w\|^2
\end{align}
\end{subequations}
is the \emph{proximal average} of $h$ and
$\iota_{\{w\}}$. 
See \cite{BGLW} and the reference therein 
for more on the proximal average. 
\end{remark}

\section{The Arag\'on Artacho--Campoy algorithm (AACA)}

\label{sec:main}

From now on, we suppose that 
\begin{empheq}[box=\mybluebox]{equation}
    \text{$A$ and $B$ are maximally monotone on $X$, $w\in X$,
    and $\gamma\in\zeroun$.}
\end{empheq}
Let $\sigma_A\geq 0$ and $\sigma_B\geq 0$ be such that
\begin{empheq}[box=\mybluebox]{equation}
\text{$A-\sigma_A\Id$ and $B-\sigma_B\Id$ are monotone,}
\end{empheq}
and we also assume that 
\begin{empheq}[box=\mybluebox]{equation}
\label{eq:A+Bmm}
    \text{$A+B$ is maximally monotone}
\end{empheq}
which will make all results more tidy.
(See also Remark~\ref{r:AACA} below.)
Next, as in Remark~\ref{r:reproav},
we introduce the resolvent averages between
$A,B$ and $N_{\{w\}}$:
\begin{empheq}[box=\mybluebox]{equation}
A_\gamma \colon X\To X\colon x\mapsto
A\big(\gamma^{-1}(x-(1-\gamma)w)\big)+\gamma^{-1}(1-\gamma)(x-w)
\end{empheq}
and 
\begin{empheq}[box=\mybluebox]{equation}
B_\gamma \colon X\To X\colon x\mapsto 
B\big(\gamma^{-1}(x-(1-\gamma)w)\big)+\gamma^{-1}(1-\gamma)(x-w).
\end{empheq}

\begin{proposition}
\label{p:0525b}
The following hold true:
\begin{enumerate}
\item
\label{p:0525bi}
$A_\gamma$, $B_\gamma$, and $A_\gamma+B_\gamma$ are maximally
monotone. 
\item 
\label{p:0525bii}
$A_\gamma$, $B_\gamma$, and $A_\gamma+B_\gamma$ are strongly monotone,
with constants
$\gamma^{-1}(\sigma_A+1-\gamma)$, 
$\gamma^{-1}(\sigma_B+1-\gamma)$, 
and $\gamma^{-1}(\sigma_A+\sigma_B+2-2\gamma)$, respectively. 
\item
\label{p:0525biii}
$\zer(A_\gamma+B_\gamma)$ is nonempty and a singleton. 
\end{enumerate}
\end{proposition}
\begin{proof}
\ref{p:0525bi}: Clear.
\ref{p:0525bii}: This follows from \eqref{e:0525a}.
\ref{p:0525biii}: 
Items \ref{p:0525bi} and \ref{p:0525bii} imply 
that $A_\gamma+B_\gamma$ is maximally monotone
and strongly monotone.
Now apply 
\cite[Corollary~23.37(ii)]{BC}. 
\end{proof}

In view of Proposition~\ref{p:0525b}\ref{p:0525biii}, 
we denote the
unique point in $\zer(A_\gamma+B_\gamma)$ by $z_\gamma$:
\begin{empheq}[box=\mybluebox]{equation}
\label{e:zgamma}
\zer(A_\gamma+B_\gamma)=\{z_\gamma\}. 
\end{empheq}

We now essentially re-derive 
the central convergence result of
Arag\'on--Artacho and Campoy \cite[Theorem~3.1]{Fran2}:

\begin{theorem}[AACA for fixed $\gamma$]
\label{t:AACA}
Given $x_0\in X$ and $\lambda\in\left]0,1\right]$,
define the sequence $(x_n)_{n\in\NN}$ via 
\begin{equation}
(\forall n\in\NN)\quad x_{n+1} = (1-\lambda)x_n+\lambda 
\big(2\gamma J_B + 2(1-\gamma)w-\Id\big)\circ
\big(2\gamma J_A + 2(1-\gamma)w-\Id\big)x_n.
\end{equation}
Then there exists $\bar{x}\in\Fix(R_{B_\gamma}R_{A_\gamma})$
such that $x_n\weakly \bar{x}$ and
$\gamma J_Ax_n+(1-\gamma)w \to z_\gamma$. 
\end{theorem}
\begin{proof}
On the one hand, by Proposition~\ref{p:resav},
\begin{equation}
J_{A_\gamma} = \gamma J_A+(1-\gamma)w
\quad\text{and}\quad
J_{B_\gamma} = \gamma J_B+(1-\gamma)w
\end{equation}
which implies
\begin{equation}
R_{A_\gamma} = 2\gamma J_A + 2(1-\gamma)w-\Id
\quad\text{and}\quad
R_{B_\gamma} = 2\gamma J_B + 2(1-\gamma)w-\Id
\end{equation}
and further
\begin{equation}
R_{B_\gamma}R_{A_\gamma}
= \big(2\gamma J_B + 2(1-\gamma)w-\Id\big)\circ
\big(2\gamma J_A + 2(1-\gamma)w-\Id\big)
\end{equation}
On the other hand,
both $A_\gamma$ and $B_\gamma$ are strongly monotone 
with constant $\gamma^{-1}(1-\gamma)$.
Altogether, the result follows from 
Fact~\ref{f:DRPR} applied to $(A_\gamma,B_\gamma)$ instead of
$(A,B)$. 
\end{proof}

\begin{remark} 
\label{r:AACA}
Several comments regarding Theorem~\ref{t:AACA}
are in order.
\begin{enumerate}
\item We have opted for a more explicit and thus easier-to-use
version of AACA where the effect of $w$ is explicitly recorded. 
\item 
While one could make $\lambda$ depending on $n$ as in
\cite{Fran2}, we decided
instead to stress the new case when $\lambda=1$, corresponding to the
Peaceman--Rachford version and notably absent in \cite{Fran2}. 
This case deserves interest because it turned out to be 
the best parameter choice in 
\cite{BBK}. 
\item
Our assumption of maximal monotonicity makes for a tidy
theory. It is used chiefly to guarantee the existence of each 
$z_\gamma$;
in \cite{Fran2}, this is replaced by some condition regarding the
existence of $z_\gamma$ which seems to be not so easy to check in
practice. 
\item
One may apply Theorem~\ref{t:AACA} in a standard 
product space setting to handle the sum of finitely many
maximally monotone operators via AACA, as done in \cite{Fran2}. 
\end{enumerate}
\end{remark}

Of course, the remaining key question is: 
\begin{center}
{\em What is the behaviour
when $\gamma\to 1^-$ for AACA?}
\end{center}
While this was answered in some form in \cite{Fran1} when $A$ and
$B$ are normal cone operators, no result was offered in
\cite{Fran2}. We conclude this paper by providing a complete and
satisfying answer, relying on tools by Combettes and Hirstoaga
\cite{CH06} and \cite{CH08}, packed also into
\cite[Theorem~23.44]{BC}. 

\begin{theorem}[dichotomy for AACA when $\gamma\to 1^-$]
\label{t:dichotomy}
Let $z_\gamma$ be as in \eqref{e:zgamma}.
Then exactly one of the following holds:
\begin{enumerate}
\item
\label{t:dichotomyi}
$\zer(A+B)\neq\varnothing$ and $z_\gamma\to P_{\zer(A+B)}w$ as
$\gamma\to1^-$. 
\item
\label{t:dichotomyii}
$\zer(A+B)=\varnothing$ and $\|z_\gamma\|\to\infty$ as $\gamma\to
1^-$. 
\end{enumerate}
\end{theorem}
\begin{proof}
Set $\delta=2(1-\gamma)$ and note that
$\delta \to 0^+$ $\Leftrightarrow$ $\gamma\to 1^-$. 
Moreover, set 
\begin{equation}
y_\delta = \gamma^{-1}\big(z_\gamma-(1-\gamma)w\big).
\end{equation}
We have, by definition of $z_\gamma$ and $y_\delta$,
\begin{align}
0\in (A_\gamma+B_\gamma)(z_\gamma) = (A+B)y_\delta +
\delta(y_\delta-w).
\end{align}
Two cases are now conceivable.

\emph{Case~1:} $\zer(A+B)\neq\varnothing$.
By \cite[Theorem~23.44(i)]{BC}, we have
\begin{equation}
\lim_{\delta\to 0^+} y_\delta = P_{\zer(A+B)}w;
\end{equation}
or equivalently, $\lim_{\gamma\to 1^-} z_\gamma =
P_{\zer(A+B)}w$. 

\emph{Case~2:} $\zer(A+B)=\varnothing$.
By \cite[Theorem~23.44(ii)]{BC}, we have
\begin{equation}
\lim_{\delta\to 0^+} \|y_\delta\| = +\infty;
\end{equation}
or equivalently, $\lim_{\gamma\to 1^-} \|z_\gamma\| =
+\infty$. 

Altogether, the proof is complete. 
\end{proof}

\begin{remark}
\label{r:dichotomy}
Here are some comments on Theorem~\ref{t:dichotomy}. 
\begin{enumerate}
\item
The information presented in
Theorem~\ref{t:dichotomy}\ref{t:dichotomyii} is new even when $A$
and $B$ are normal cone operators as in \cite{Fran1}. 
\item
Computing $P_{\zer(A+B)}w$ via Theorem~\ref{t:dichotomy} is
cumbersome and ``doubly iterative'':
one must first employ an algorithm to find $z_\gamma$, and the
let $\gamma$ tend to $1^-$.
There are, however, some results that allow us to avoid this double
iteration and instead solve the problem via a single iteration;
see, e.g.,
the discussion in \cite[Section~8]{BBHM}. 
\end{enumerate}
\end{remark}

Let us conclude with a simple example.

\begin{example}
Suppose that $A=P_U$, where $U$ is a closed linear subspace of
$X$, and $B\equiv -v$, where $v\in U^\perp$. 
Then $\zer(A+B) = U^\perp$, if $v=0$;
$\zer(A+B)=\varnothing$, if $v\neq 0$. 
Let $w=0\in X$, and let 
$\gamma\in\left]0,1\right[$.
Then $(\forall x\in X)$ $A_\gamma x = \gamma^{-1}(P_U(x)+(1-\gamma)x)$ and
$B_\gamma x = -v+\gamma^{-1}(1-\gamma)x$. 
Hence $\zer(A_\gamma+B_\gamma) = \{z_\gamma\}$, where
$z_\gamma = (2(1-\gamma))^{-1}\gamma v$. 

\emph{Case~1}: $v=0$.
Then $z_\gamma\equiv 0 \to 0 = P_{\zer(A+B)}(w)$.

\emph{Case~2}: $v\neq 0$.
Then $\|z_\gamma\|=(2(1-\gamma))^{-1}\gamma\|v\|\to+\infty$.

Both cases illustrate Theorem~\ref{t:dichotomy}. 
\end{example}

\section*{Acknowledgments}{   
	SA was partially supported by 
	Saudi Arabian Cultural Bureau. 
	HHB and XW were partially 
	supported by NSERC Discovery Grants. 
	WMM was partially supported by NSERC Postdoctoral
	Fellowship. 
}


\end{document}